%% file: DBC-NYJ-revised.tex
\documentclass[11pt]{amsart}
\usepackage{hyperref}
\hypersetup{nesting=true,debug=true,naturalnames=true}
\usepackage{graphicx,amssymb,upref}

\let\<\langle
\let\>\rangle

\let\uml\"
\usepackage[dvipsnames,usenames]{color}
\usepackage{float}
\usepackage{amsmath, amsthm, amssymb}
\usepackage{amsfonts}
\usepackage{enumerate}
\usepackage{epsf}
\usepackage{ifpdf}
\usepackage{relsize}
\usepackage{amsmath}
\usepackage[all]{xy}
\usepackage{graphicx}
\usepackage{epsfig}
\usepackage{epstopdf}
\usepackage[UKenglish]{babel}

\newtheorem{thm}{Theorem}[section]
\newtheorem{lemma}[thm]{Lemma}
\newtheorem{prop}[thm]{Proposition}
\newtheorem{construct}[thm]{Construction}
\newtheorem{cor}[thm]{Corollary}

\newtheorem*{claim}{Claim}

 \theoremstyle{definition}
\newtheorem{defn}[thm]{Definition}
\theoremstyle{plain}
\theoremstyle{remark}
\newtheorem{rmk}[thm]{Remark}

\numberwithin{equation}{section}
\numberwithin{enumi}{subsection}
\newcommand{\spin}{\ifmmode{\rm Spin}\else{${\rm spin}$\ }\fi}
\newcommand{\spinc}{\ifmmode{{\rm Spin}^c}\else{${\rm spin}^c$}\fi}

\newcommand{\D}{\mathcal{D}}

\makeatletter

\newcommand{\Rmnum}[1]{\expandafter\@slowromancap\romannumeral #1@}
\title{On L-space knots obtained from unknotting arcs in alternating diagrams}
\author[Andrew Donald]{Andrew Donald}
\address {School of Mathematics, University of Bristol, Bristol, BS8 1TW, UK}
\email{\href{mailto:andrew.donald@bristol.ac.uk}{andrew.donald@bristol.ac.uk}}

\author[Duncan McCoy]{Duncan McCoy}
\address {Department of Mathematics, University of Texas at Austin, Austin, TX 78712}
\email{\href{mailto:d.mccoy@math.utexas.edu}{d.mccoy@math.utexas.edu}}

\author[Faramarz Vafaee]{Faramarz Vafaee}
\address {Department of Mathematics, Duke University, Durham, NC 27708}
\email{\href{mailto:vafaee@math.duke.edu}{vafaee@math.duke.edu}}



\begin{document}

\begin{abstract}
Let $D$ be a diagram of an alternating knot with unknotting number one. The branched double cover of $S^3$ branched over $D$ is an L-space obtained by half integral surgery on a knot $K_D$. We denote the set of all such knots $K_D$ by $\mathcal D$.
We characterize when $K_D\in \mathcal D$ is a torus knot, a satellite knot or a hyperbolic knot.
In a different direction, we show that for a given $n>0$, there are only finitely many L-space knots in $\D$ with genus less than $n$.
\end{abstract}
\maketitle
\tableofcontents
\section{Introduction}\label{sec:intro}
A knot $K\subset S^3$ is an {\it L-space knot} if it admits a positive Dehn surgery to an L-space.\footnote{An {\it L-space} $Y$ is a rational homology sphere with the simplest possible Heegaard Floer invariant, that is $\text{rk }\widehat{HF}(Y)=|H_1(Y;\mathbb Z)|$.} Examples include torus knots, and more broadly, Berge knots in $S^3$~\cite{Berge}. In recent years, work by many researchers provided insight on the fiberedness~\cite{Ni2009,Ghiggini}, positivity~\cite{Hedden2010}, and various notions of simplicity of L-space knots~\cite{Ath,HeddenBerge,Krcatovich2014}. For more examples of L-space knots, see~\cite{Hom2011a, Hom2016, HomLidmanVafaee, Motegi2014, Motegi2014b, Vafaee2013}.

Since the double branched cover of an alternating knot is an $L$-space \cite{OS-BDC2005}, the Montesinos trick allows us to construct a knot with an $L$-space surgery for every alternating knot with unknotting number one \cite{montesinos1973,Ozsvath2005a}. The primary focus of this paper is to study the family of L-space knots arising in the branched double cover of alternating knots with unknotting number one.

Let $(D,c)$ be an alternating knot diagram $D$ with an unknotting crossing $c$. By cutting out the interior of a small ball containing $c$ and taking the branched double cover we obtain the complement of a knot in $S^3$. We call this $K_{(D,c)}$, or $K_D$ if we implicitly assume that an unknotting crossing has been chosen. We also call the arc connecting the two arcs of the unknotting crossing $c$ of $D$ the \emph{unknotting arc}. 

Let $\D$ denote the set of $K_{(D,c)}$ obtained by considering all reduced alternating diagrams with unknotting number one. Here, a diagram is {\em reduced} if it does not contain any nugatory crossings (see Figure~\ref{fig:nugatory}).
 \begin{figure}[t]
\centering
\includegraphics{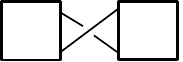}
\put(-78,11.5){\large $T_1$}
\put(-22.5,11.5){\large $T_2$}
\caption{A nugatory crossing.}
\label{fig:nugatory}
\end{figure}

\begin{rmk} It is worth noting that a diagram with unknotting number one is not enough to determine a knot in $\D$ on its own, and one really does need to specify the unknotting crossing. For example, any alternating diagram for the knot $8_{13}$ possesses both a positive and a negative unknotting crossing. These give rise to different knots in $\D$, namely the torus knots $T_{2,7}$ and $T_{3,-5}$.
\end{rmk}

By the work of Thurston~\cite{Thurston1982}, any knot in $S^3$ is precisely one of a torus knot, a satellite knot or a hyperbolic knot. In this paper, we characterize when each of these knot types arise in $\mathcal D$. When $K_D$ is a satellite knot, its exterior $S^3\setminus \nu(K_D)$ contains an incompressible, non-boundary parallel torus. (Here, $\nu(K_D)$ indicates an open tubular neighborhood of $K_D$ in $S^3$.) Correspondingly, there will be a Conway sphere $C$ in $D$.

\begin{defn}
Let $(D,c)$ be an alternating diagram with an unknotting crossing $c$. Let $C$ be a visible\footnote{A Conway sphere is {\em visible} in a diagram if it intersects the plane of the diagram in a single simple closed curve. A Conway sphere that is not visible, is {\it hidden} (cf. \cite[Section~3]{Thistlethwaite1991}).} Conway sphere in $D$, disjoint from the unknotting arc specified by $c$. We will call the component of $S^3\setminus C$ containing $c$ the {\em interior} of $C$. We will call the other component the {\em exterior}. We say that $C$ is {\em substantial} if the interior of $C$ contains more than one crossing and the exterior is not a rational tangle.
\end{defn}
\begin{rmk}
A substantial Conway sphere can be non-essential, since the interior may be a rational tangle.
\end{rmk}
We may now state the main result of the paper.
\begin{thm}\label{thm:satellitecharacterization}
Let $(D,c)$ be an alternating diagram with an unknotting crossing $c$. The following are equivalent:
\begin{enumerate}[(i)]
\item $K_D$ is a satellite,
\item $D$ contains a substantial Conway sphere.
\end{enumerate}
\end{thm}

If $D$ is a 2-bridge knot diagram, then we show that $K_D$ is a torus knot. Conversely, it turns out that this is the only way that torus knots arise in $\D$. See~\cite[Theorem~1]{Murakami1986} for a characterization of 2--bridge knots with unknotting number one.
\begin{prop}\label{lem:toruscharacterization}
The knot $K_D$ is a torus knot if and only if $D$ is a diagram of a 2-bridge knot.
\end{prop}

Combining Proposition~\ref{lem:toruscharacterization} and Theorem~\ref{thm:satellitecharacterization} allows us to determine for which diagrams $K_D$ is a hyperbolic knot. (See also~\cite{Gordon2004non,Gordon2006knots} for relevant results.)
\begin{cor}\label{cor:HyperbolicCharacterization}
Let $(D,c)$ be an alternating diagram with an unknotting crossing $c$. Then $K_D$ is a hyperbolic knot if and only if $D$ is not a 2-bridge knot diagram and $D$ does not contain a substantial Conway sphere.
\end{cor}

Conjecturally, there are only finitely many $L$-space knots in $S^3$ with a given genus (\cite[Conjecture~6.7]{Hedden2014} and~\cite[Conjecture~1.2]{Baker2015}). As a final result, we verify this conjecture for knots in $\D$.
\begin{prop}\label{prop:genus}
For a given $n>0$ there are finitely many knots in $\D$ with genus less than $n$.
\end{prop}
\begin{proof}
Let $K_D\in \D$ correspond to a diagram $(D,c)$. Using Montesinos trick~\cite{montesinos1973}, we have $S_\frac{d}{2}^3(K_D)\cong \Sigma(D)$, for some $d$ with $|d|=\det D$. Thus, by \cite[Theorem~1.1]{mccoy2014bounds}, we have $\det D /2 \le 4g(K_D) +3$. The result follows since there are finitely many alternating knots of a given determinant and each alternating knot has finitely many reduced alternating diagrams up to planar isotopy.
\end{proof}
The rest of the paper is organized as follows. In Section~\ref{Tsukamoto}, we state Tsukamoto's theorem~\cite[Theorem~5]{Tsukamoto2009} that gives a set of three moves ({\it flype, tongue}, and {\it twirl}) enabling us to go from a clasp diagram (Figure~\ref{fig:claspdiagram}) to any diagram $(D,c)$. We also determine the corresponding effect of these moves to the knots in $\mathcal D$. Theorem~\ref{thm:satellitecharacterization} and Proposition~\ref{lem:toruscharacterization} are proved in Section~\ref{SatelliteClassification}.\\

\noindent{\bf Acknowledgements}. We would like to thank Ken Baker, Josh Greene, Matt Hedden, John Luecke and Tom Mark for helpful conversations. We are also grateful to an anonymous referee for their detailed feedback.

\section{Almost alternating diagrams of the unknot}\label{Tsukamoto}

Recall that a diagram $D$ is said to be {\em almost alternating} if it is non-trivial, non-alternating and can be changed into an alternating diagram by changing a single crossing. Given an almost alternating diagram we call a crossing which can be changed to obtain an alternating diagram a {\em dealternator}. We will refer to the crossing arc of a dealternator as a {\em dealternating arc}.

As the following remark shows, for every diagram $(D,c)$ there is a reformulation of $K_{(D, c)}$ in terms of dealternating arcs. We will sometimes find it convenient to use this alternative approach.

{\rmk Given an alternating diagram with an unknotting crossing $(D,c)$, let $\widetilde{D}$ be the almost alternating diagram of the unknot obtained by changing $c$. Since $\widetilde{D}$ is unknotted, taking the double branched cover lifts its dealternating arc to a knot in $S^3$, which can easily seen to be $K_{(D,c)}$.}

\begin{figure}[h]
\centering
\def\svgwidth{13cm}
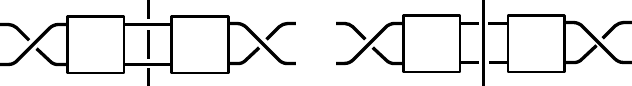
\caption{Flyped tongues. In both pictures the dealternator is the top central crossing.}
\label{fig:flypedtongue}
\end{figure}

Tsukamoto has studied the structure of almost alternating diagrams of the unknot. The key result we require shows that most almost alternating diagrams of the unknot contain a flyped tongue, where a {\em flyped tongue} is a region appearing as in Figure~\ref{fig:flypedtongue} \cite[Theorem~4]{Tsukamoto2009}.

{\thm [Tsukamoto]\label{thm:Tsukflypedtongue}
Any reduced almost alternating diagram of the unknot which is not an unknotted clasp diagram (see Figure~\ref{fig:claspdiagram}) contains a flyped tongue.}
\begin{figure}[h]
  \centering
  \def\svgwidth{4cm}
  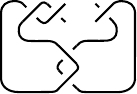
 \caption{An unknotted clasp diagram}
 \label{fig:claspdiagram}
\end{figure}

We recall some conventions on tangles. For the purposes of this paper a tangle is a pair of unoriented arcs properly embedded in $B^3$ with the end points are mapped to four marked points on $\partial B^3$. We consider tangles up to isotopy fixing $\partial B^3$. These will usually arise as the intersection of a knot diagram with a ball bounded by a visible Conway sphere. In this setting, a tangle is said to be {\em crossingless} if the region of the diagram contained in the ball is without crossings, as shown in Figure~\ref{fig:crossingless}. We call a tangle which is equivalent to a crossingless tangle a {\em trivial} tangle. If a tangle can be isotoped to a trivial tangle by an isotopy of $B^3$ (with no requirement that $\partial B^3$ is fixed), then we call it a {\em rational tangle}.
\begin{figure}[h]
  \centering
  \def\svgwidth{2cm}
  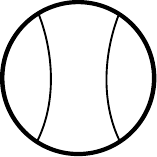
 \caption{A crossingless tangle}
 \label{fig:crossingless}
\end{figure}


A reduced almost alternating diagram which contains a flyped tongue can be isotoped to an almost alternating diagram with fewer crossings by twisting the two tangles so that the outermost two crossings are removed from the diagram. If the result of this isotopy is not reduced, then we can further perform a Reidemeister I move to obtain a reduced diagram. Given an almost alternating diagram of the unknot, Theorem~\ref{thm:Tsukflypedtongue} then allows us to carry out a sequence of such isotopies. Carrying out this argument more carefully allows one to show that all reduced alternating diagrams with an unknotting crossing can be built up by a sequence of simple local moves. The following is immediate from~\cite[Theorem~5]{Tsukamoto2009} (see also \cite[Theorem~2]{mccoy17unknotting}).

{\thm [Tsukamoto] \label{thm:altTsuka} Let $(D,c)$ be a reduced alternating diagram with an unknotting crossing. Then there are a sequence of reduced alternating diagrams $D_i$ with unknotting crossings $c_i$,
\[(D_1,c_1) \to \dots \to (D_p,c_p)=(D,c),\]
such that $(D_1,c_1)$ is a clasp diagram (see Figure~\ref{fig:claspdiagram}), and for each $i$, $(D_i,c_i)$ is obtained from $(D_{i-1},c_{i-1})$ by either a flype fixing $c_{i-1}$, a twirl move, or a tongue move (see Figure~\ref{fig:alttwirltongue}).}

\begin{figure}[t]
\centering
\def\svgwidth{7cm}
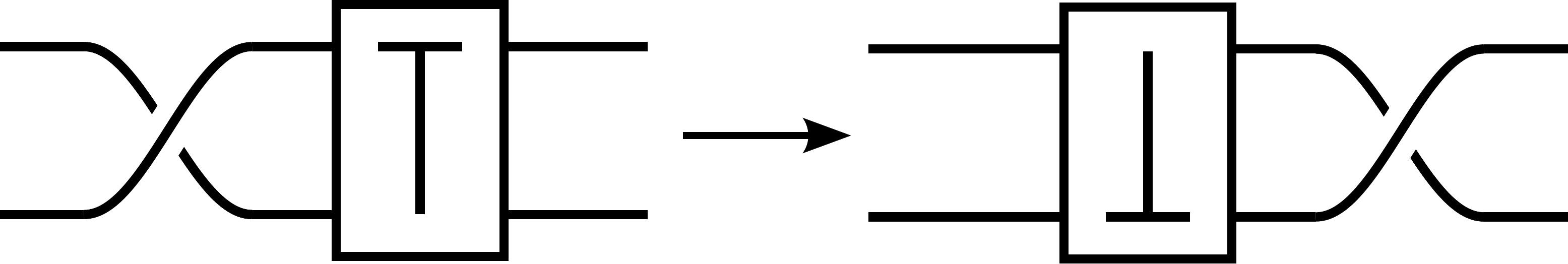
\caption{A flype move}
\label{flype}
\end{figure}

\begin{figure}[t]
  \centering
  \def\svgwidth{7cm}
  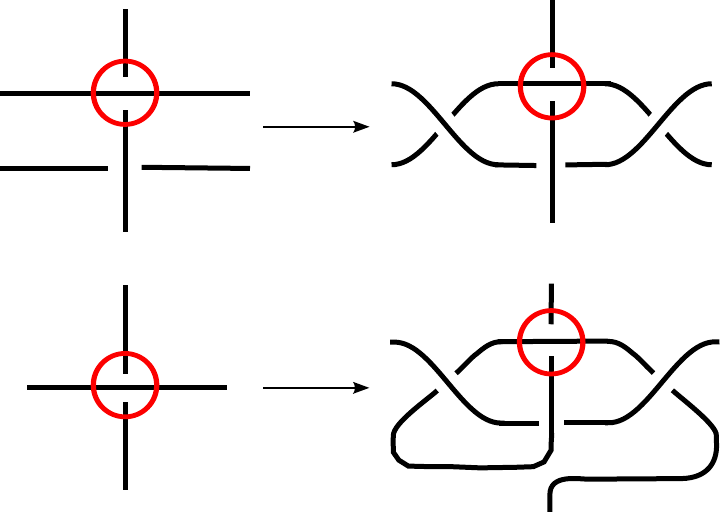
 \caption{Up to reflection, the tongue move (top) and a twirl move (bottom). In each case the new unknotting crossing is marked by the red circle.}
\label{fig:alttwirltongue}
\end{figure}

\begin{figure}[t]
\centering
\def\svgwidth{10cm}
  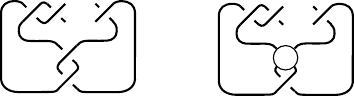
\caption{A clasp diagram (left) and the rational tangle obtained by excising a ball containing an unknotting crossing (right).}
\label{fig:clasp}
\end{figure}
\subsection{The corresponding operations in $\D$}\label{sec:OperationsInD}
In this section, we turn our focus to the moves, used in Theorem~\ref{thm:altTsuka}, that allow us to build up any alternating diagram with unknotting number one. In principle, by studying the effect of these moves on the knots in $\D$, one should be able to obtain an understanding of the knots arising in $\D$. In practice, however, this approach is not particularly straightforward to carry out. As we will see, the effects of a twirl move can be understood easily enough, but the effects of a tongue move are much more subtle.

The diagrams from which Theorem~\ref{thm:altTsuka} allows us to build all other alternating knots with unknotting number one are the clasp diagrams. According to the following proposition, the clasp diagrams are precisely those which give rise to the unknot in $\D$.

\begin{prop}\label{UnknotDistinguish}
$K_D$ is the unknot if and only if $D$ is a clasp diagram.
\end{prop}
\begin{proof}
For any odd integer $d$, the manifold $S_{d/2}^3(U)$ is homeomorphic to the lens space $L(d,2)$. Thus if $K_D=U$, then $D$ must be a clasp diagram~\cite[Corollary~4.12]{Hodgson1985}. Conversely, as shown in Figure~\ref{fig:clasp}, when a ball containing an unknotting crossing is cut out, we obtain a rational tangle. Thus if $D$ is a clasp diagram then the complement of $K_D$ is a solid torus and hence $K_D$ is the unknot.
\end{proof}

{\rmk Note that as the flypes in Theorem~\ref{thm:altTsuka} can be chosen to fix the unknotting crossing, they will clearly leave the knot $K_D$ unchanged.}\\

A twirl move corresponds to a cabling operation. It is a special case of a more general operation for producing satellite knots in $\D$ studied in the following section.
\begin{prop}\label{prop:twirleffect}
If $(D',c')$ is obtained from $(D,c)$ by a twirl move, then $K_{D'}$ is obtained from $K_D$ by taking a cable with winding number two.

\end{prop}
\begin{proof}
This can be seen from Figure~\ref{Fig:twirl} which shows how the dealternating arc changes under introducing a twirl.\footnote{We may construct $K_D$ from $(D,c)$ as follows. First, we change the crossing $c$ in $D$, and isotope the resulting diagram to obtain the unknot, while we keep track of the unknotting arc throughout the isotopies. Second, we double the resulting arc to obtain $K_D$. See~\cite{Hodgson1985}, for instance.} After performing the isotopy to remove the twirl it is clear that the blue arc can be isotoped onto the sphere $S$ whose intersection with the plane is the circle around the crossing. This shows that after taking the double branched cover, $K_D'$ can be isotoped to lie on the torus obtained by taking the double branched cover of $S$. As the lift of $S$ is the boundary of a tubular neighborhood of $K_D$, this shows that $K_D'$ is a cable of $K_D$. It is straightforward to see that the winding number of this cable is two by considering a curve on $S$ which lifts to a meridian of $K_D$.
\end{proof}

\begin{figure}[ht]
\centering
\def\svgwidth{7cm}
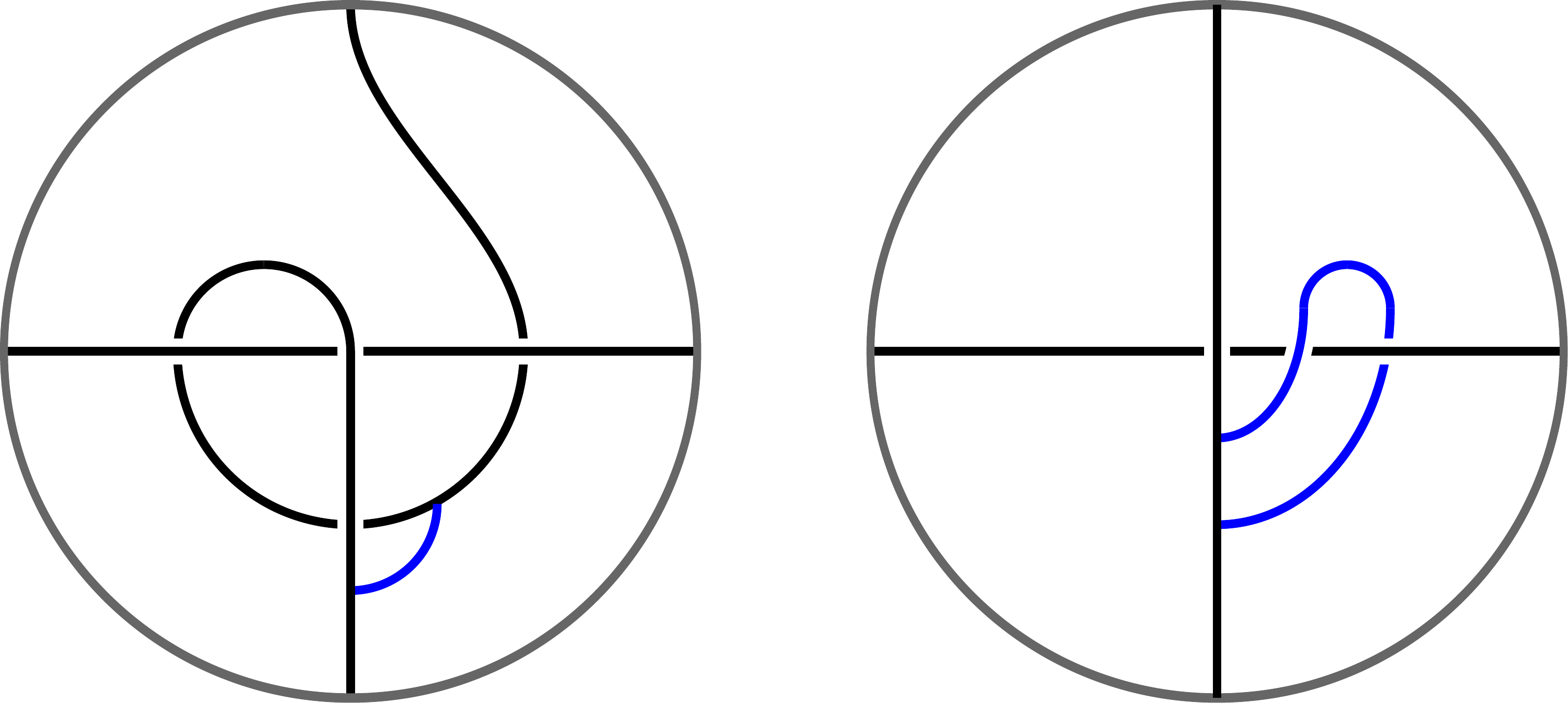
\caption{
{The diagram on the left is obtained from a twirl move on the right one. The blue arc shows how the unknotting arc in $D'$ appears in $D$ after reversing the twirl move.}}
\label{Fig:twirl}

\end{figure}

Finally we turn our attention to the tongue move. As we will see this corresponds to a band sum operation in $\D$. In certain cases, it can be described precisely what the band sum is, however, in general, different tongue moves on the same diagram will result in different knots in $\D$. For example, if $D$ is an alternating diagram of $8_{14}$ there is an unknotting crossing for which $K_D=T_{3,5}$. There is a tongue move to a ten-crossing diagram $D'$ with $K_{D'}=T_{4,7}$ and a tongue move to a different ten-crossing diagram $D''$ for which $K_{D''}$ is a positive braid with braid index equal to six.

\begin{prop}\label{prop:tongue=bandsum}
If $(D',c')$ is obtained from $(D,c)$ by a tongue move, then $K_{D'}$ is obtained by taking a band sum of $K_D$ with another knot.
\end{prop}
\begin{proof}
This follows from  Figure~\ref{fig:BandSum}. It is clear that the lift of the new dealternating arc is obtained by taking a band sum of knots obtained by lifting the blue and red arcs to the double branched cover. Note that the lift of the blue arc is $K_D$.
\end{proof}
\begin{figure}[ht]
\centering

\def\svgwidth{7cm}
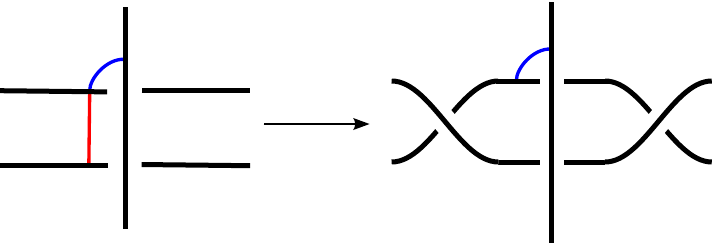
\caption{A tongue move corresponds to a band sum operation in $\mathcal D$.}
\label{fig:BandSum}

\end{figure}

\section{The geometry of knots in $\D$}\label{SatelliteClassification}
In this section we study which types of L-space knots arise in $\D$. Precisely, we show when $K_D$ is a torus knot, a satellite knot or a hyperbolic knot.

\subsection{Torus knots in $\D$}
We have already seen from Proposition~\ref{UnknotDistinguish} when the unknot arises in $\D$. We now generalize this to show that torus knots correspond to 2-bridge knots with unknotting number one.

\begin{prop}
The knot $K_D$ is a torus knot if and only if $D$ is a diagram of a 2-bridge knot.
\end{prop}
\begin{proof}
If $S_{d/2}^3(T_{r,s})\cong \Sigma(L)$ is the double branched cover of an alternating knot $L$, then $d=2rs\pm 1$ \cite{mccoy2014bounds}. In particular, $\Sigma(L)$ must be the lens space $L(2rs\pm 1, 2r^2)$ \cite{Moser71elementary}. Since the only knots in $S^3$ with lens space branched double covers are the 2-bridge knots~\cite[Corollary~4.12]{Hodgson1985}, this shows that if $K_D$ is a torus knot, then $D$ is a diagram of a 2-bridge knot.

Conversely, if $D$ is an alternating diagram of a 2-bridge knot, then there is a Conway sphere $C$ passing through the unknotting crossing, such that the both components of $D$ in $S^3\setminus C$ are rational tangles (cf. Figure~\ref{fig:2-bridgeimpliestorus}). Consider the diagram of the unknot obtained by changing the unknotting crossing in $C$. Since both sides of $C$ contain a rational tangle, we see that $C$ lifts to an unknotted torus in $S^3$ (it bounds a solid torus on both sides) upon taking the double branched cover. Since the unknotting arc in $D$ can be isotoped to lie on $C$, $K_D$ must lie on an unknotted torus in $S^3$. In particular, it is a torus knot, as required.
\end{proof}
\begin{rmk}
One can also appeal to the Cyclic Surgery Theorem \cite{Culler1987CyclicSurgery} and~\cite[Theorem~1]{Murakami1986} to show that a 2-bridge diagram yields a torus knot.
\end{rmk}

\begin{figure}[h]
\centering
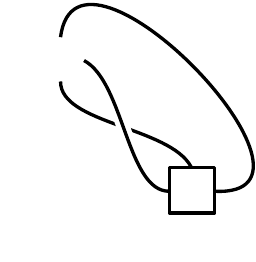
\caption{An unknotting number one 2-bridge diagram with a Conway sphere (marked in red) onto which the unknotting arc can be isotoped. Both $A$ and $B$ are rational.}
\label{fig:2-bridgeimpliestorus}
\end{figure}
{\rmk By exhibiting a 2-bridge diagram along with an unknotting crossing for each torus knot, it can be proved that all torus knots are in $\mathcal D$.}

\subsection{Constructing satellite knots in $\D$}
Now we turn our attention to the satellite knots in $\D$. In this section we give a general construction for producing alternating diagrams with an unknotting crossing for which the resulting knot in $\D$ is a satellite knot. For this construction we need the following definition.
\begin{defn}\label{UnknottingTangle}
Let $T$ be an alternating tangle which does not admit any Reidemeister~I moves reducing the crossing number. We say that $T$ is an {\em alternating unknotting tangle} if there is a distinguished crossing $c_T$ such that after changing $c_T$ we obtain a tangle that is isotopic, relative to the boundary, to a single crossing. Moreover, we require that the crossings connected by an arc to the boundary appear as in Figure~\ref{fig:unknottingtangle}.
\end{defn}
\begin{figure}[h]
\centering
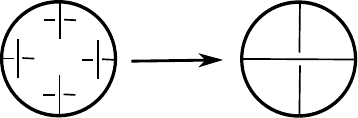
\caption{The left picture is a diagram of an alternating unknotting tangle showing only the crossings connected by an arc to the boundary. The right picture shows the tangle obtained by changing the distinguished crossing.}
\label{fig:unknottingtangle}
\end{figure}
Note that an alternating unknotting tangle $T$ is defined in such a way that if we replace the unknotting crossing of $(D,c)$ with $T$ to obtain a new alternating diagram, then this alternating diagram will be reduced and have unknotting number one with the distinguished crossing $c_T$ as an unknotting crossing. This observation allows us to make the following construction possible.

\begin{construct}\label{cons:satellite}
Let $D$ be a reduced alternating diagram with an unknotting crossing $c$, which is not a clasp diagram. Let $(T,c_T)$ be an alternating unknotting tangle with more than one crossing. Replace $c$ with $T$ to obtain an alternating diagram $(D',c_T)$ with unknotting number one.
\end{construct}
Note that the diagram resulting from this construction is always reduced. This is a consequence of the assumption that the original diagram $D$ is reduced and that the tangle $T$ cannot be simplified by any Reidemeister~I moves.
An example of this construction is given by the twirl move which replaces the unknotting crossing with an alternating unknotting tangle with four crossings. Proposition~\ref{prop:twirleffect} shows that the twirl move corresponds to a cabling operation in $\D$. Similarly Construction~\ref{cons:satellite} corresponds to taking satellites in $\D$.

\begin{lemma}
If $(D', c')$ is an alternating diagram with an unknotting crossing obtained by Construction~\ref{cons:satellite}, then $K_{D'}$ is a satellite knot.
\end{lemma}

\begin{proof}
Let  $(D', c')$ be obtained by Construction~\ref{cons:satellite} from an alternating diagram $D$ and an alternating unknotting tangle $(T, c_T)$. Recall that, by definition, the knot complement $S^3 \setminus \nu (K_{D'})$ is obtained by removing a ball containing the unknotting crossing $c_T$ and taking the double branched cover. Let $C$ be the Conway sphere bounding the alternating unknotting tangle $T$ in $D'$. The sphere $C$ lifts to a torus $R$ in the knot complement $S^3 \setminus \nu (K_{D'})$. On the side of $R$ that is branched over the exterior of $C$, we have the knot complement $S^3\setminus \nu(K_D)$. As we are assuming $D$ is not a clasp diagram, Lemma~\ref{UnknotDistinguish} shows that $K_D$ is not the unknot and, as a result, $S^3\setminus \nu(K_D)$ is not a solid torus: the unknot is the only knot whose complement is a solid torus. On the side of $R$ obtained by branching over the interior of $C$, we have $X_T$, the space  obtained by taking the branched double cover over $T$ after excising a ball containing $c_T$. Since changing $c_T$ in $T$ results in a tangle isotopic to one with a single crossing, $X_T$ is the complement of a knot $\kappa_T$ in $S^1 \times D^2$. Moreover, as we are considering $K_D$ and $K_{D'}$ as knots in $S^3$, the identification of $X_T$ as a knot complement in $S^1 \times D^2$ is such that the boundary of a disk in $S^1 \times D^2$ corresponds to a meridian of $K_D$ in $R$.

Thus $K_{D'}$ is obtained by a satellite operation with companion $K_D$ and pattern $\kappa_T$ (for some choice of longitude of $S^1 \times D^2$). It remains to show that the satellite operation is not trivial. That is, the torus $R$ is incompressible and not boundary-parallel in $S^3 \setminus \nu (K_{D'})$. Equivalently, we need to show that the knot $\kappa_T$ cannot be isotoped to lie in a ball in $S^1 \times D^2$ and is not isotopic to the core of $S^1 \times D^2$.

Consider the sequence of alternating diagrams constructed by iterating Construction~\ref{cons:satellite} with the tangle $T$. That is, let $D_0=D$ and for each $n$ let $D_n$ be the alternating diagram obtained by replacing the unknotting crossing in $D_{n-1}$ with a copy of $T$.

Since $T$ is assumed to contain more than just a single crossing, the crossing number of $D_n$ is monotonically increasing with $n$ and, by construction, all of these diagrams are reduced. Hence we have that $\det D_n \rightarrow \infty$ as $n\rightarrow \infty$. By combining the Montesinos trick~\cite{montesinos1973} with \cite[Theorem~1.1]{mccoy2014bounds} as in the proof of Proposition~\ref{prop:genus} we obtain the bound
\[\det D_n \le 8g(K_{D_n}) +6.\]
This shows that the genera $g(K_{D_n})$ are unbounded.

By the discussion from the start of the proof, $K_{D_n}$ is obtained by taking a satellite with companion $K_{D_{n-1}}$ and pattern $\kappa_T$. If $\kappa_T$ were isotopic to the core of $S^1 \times D^2$, then we would have $K_{D_{n}}= K_{D}$ for all $n$. If $\kappa_T$ could be isotoped to lie in a ball in $S^1\times D^2$, then for all $n\geq 1$ we have $K_{D_{n}}$ is isotopic to $\kappa_T$ considered as a knot in $S^3$. In either case the genera $g(K_{D_n})$ would be bounded, which we have already shown to be impossible. Thus $K_{D'}$ is a satellite knot as required.

\end{proof}
The following lemma will be useful for finding alternating unknotting tangles inside an alternating diagram with unknotting number one. For this lemma we need to know what it means for the exterior of a Conway sphere to be reduced.
\begin{defn}\label{defn:reducedtangle}
Given a reduced alternating diagram with unknotting crossing $c$ in the interior of a visible Conway sphere $C$, we say that the exterior of $C$ is reduced if there is no flype which decreases the number of crossings in the exterior of $C$.
\end{defn}
An example to illustrate Definition~\ref{defn:reducedtangle} is given in Figure~\ref{fig:maximaldiagram}.
\begin{figure}
  \centering
  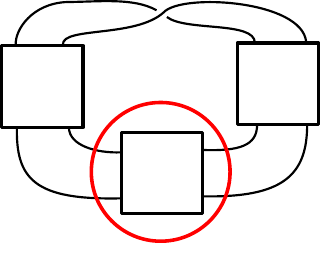
 \caption{A diagram with a Conway sphere which does not satisfy the conditions of Definition~\ref{defn:reducedtangle}. By twisting either $T_1$ or $T_2$, we can flype to move a crossing from the exterior of $C$ to the interior.}
 \label{fig:maximaldiagram}
\end{figure}
\begin{rmk} There is a more general notion of a reduced tangle due to  Thistlethwaite \cite[Definition~2.2]{Thistlethwaite1991}, which includes the definition we are using here.
\end{rmk}
\begin{lemma}\label{lem:maxitangles}
Let $D$ be an alternating diagram with unknotting crossing $c$ in the interior of a visible Conway sphere $C$. If the exterior of $C$ is reduced and contains at least one crossing, then the interior of $C$ is an alternating unknotting tangle with the distinguished crossing given by $c$.
\end{lemma}

\begin{proof}
As the details of this proof are somewhat technical, we give an overview of the strategy before we begin. We will prove the lemma by induction on the number of crossings in the interior of $C$. When the interior of $C$ contains only a single crossing there is nothing to prove. So we can assume that the interior of $C$ contains more than just the unknotting crossing. Let $D'$ be the almost-alternating diagram obtained by changing $c$. The inductive step will be achieved by finding an isotopy fixing the exterior of $C$ which produces a reduced almost-alternating diagram $D''$ with fewer crossings in the interior of $C$. By Theorem~\ref{thm:Tsukflypedtongue} there is a flyped tongue in $D'$. Let $E$ be a simple closed curve surrounding the flyped tongue which intersects the diagram in two points and two crossings as shown in Figure~\ref{fig:labelledtongue}. By carefully considering how $C$ can intersect with $D'$ and $E$, we will show that either $C$ is contained in one of the tangles $T_1$ or $T_2$ or $C$ contains $E$ in its interior. The details of this step take up the first two claims of the proof. Once we understand how $C$ sits inside $D'$ we produce the desired isotopy by flyping the tangles $T_1$ and $T_2$ before performing either an untongue move or an untwirl move. The details of this are contained in the final claim of the proof. We finish the proof by describing how this isotopy allows us to complete the inductive step.

First note that the conditions on $C$ guarantee that $D$ is not a clasp diagram: whenever we have a visible Conway sphere in a clasp diagram with at least one crossing in its exterior, we can always find a flype which carries one of the crossings into interior.

It is convenient to give the diagram $D'$ a checkerboard coloring. We may shade the diagram so that it appears as in Figure~\ref{fig:labelledtongue} near the flyped tongue.
\begin{figure}[h]
\centering
\def\svgwidth{7cm}
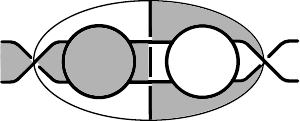
\caption{The flyped tongue in $D'$.}
\label{fig:labelledtongue}
\end{figure}

Observe that since $D'$ does not contain any nugatory crossings the four regions of $D'$ through which $E$ passes are all distinct.

Observe also that $C$ intersects four shaded regions of $D'$. Clearly, two of these regions must be white and two must be black. We see also that these four regions must all be distinct. If $C$ were to intersect the same region twice, then, up to choice of coloring, the exterior of $D'$ would appear as in Figure~\ref{fig:fourcolours}. As $D$ is alternating reduced with an unknotting crossing, it is a prime diagram, implying that both of the boxed tangles in Figure~\ref{fig:fourcolours} must be unknotted arcs. This would contradict the assumption that the exterior of $C$ contains at least one crossing.
\begin{figure}[h]
\centering
\def\svgwidth{3cm}
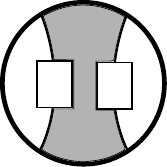
\caption{$C$ must intersect four distinct regions.}
\label{fig:fourcolours}
\end{figure}
\begin{claim}
There is a choice of $E$ which is disjoint from $C$.
\end{claim}
\begin{proof}[Proof of Claim]
Suppose that the intersection between $C$ and $ E$ is non-empty. The intersection points between $C$ and $E$ divides the intersection of $C$ with the plane of the diagram into a collection of embedded arcs $\gamma_1, \dots, \gamma_{2n}$ each disjoint from $E$ except at its endpoints.

First, let $\gamma_i$ be an arc which does not intersect $D'$. Such a $\gamma_i$ must lie within a single shaded region of $D'$. The end points of $\gamma_i$ cut $E$ into two arcs, one of which lies entirely in a single region of $D'$. If we replace an neighbourhood of the arc of $E$ lying entirely in one region with an appropriate push-off of $\gamma_i$, then we can obtain a new simple closed curve $E'$ containing the flyped tongue, but no longer intersecting $\gamma_i$ (see Figure~\ref{fig:modifyingE}). By iterating this construction, we are free to assume that every arc $\gamma_i$ intersects $D'$ at least once.
\begin{figure}[h]
\centering
\def\svgwidth{10cm}
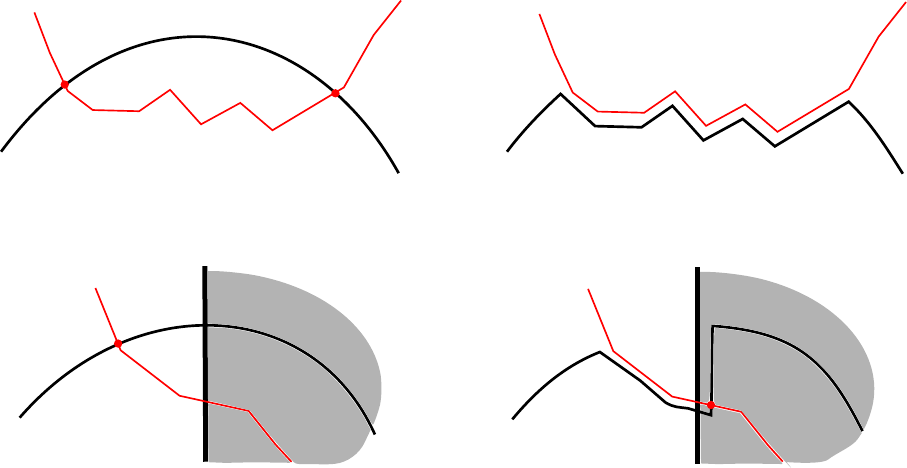
\caption{Modifying $E$ to eliminate arcs disjoint from $D'$ (top) and triangles (bottom).}
\label{fig:modifyingE}
\end{figure}
This implies that there is at least one intersection between $C$ and $D'$ that is not in the interior $E$, and, in particular, there are at most three points of intersections between $C$ and $D'$ in the interior of $E$.

Now suppose that $\gamma_i$ is an arc lying in the interior of $E$ which forms a triangle with $D'$ and $E$, i.e we have an arrangement as shown in the second row of Figure~\ref{fig:modifyingE}. We can obtain a new curve $E'$ containing the flyped tongue by replacing the segment of $E$ forming a side of the triangle with an appropriate pushed-off copy of the other two sides, as shown in Figure~\ref{fig:modifyingE}. Note that no arc $\gamma_j$ passes through the side of this triangle given by $D'$, as this would violate the condition that $C$ intersects four distinct shaded regions. Thus, $E'$ does not intersect $C$ in any more points than $E$ did. Note that if $\gamma_i$ intersected $D'$ only once, then it will not intersect $D'$ in the interior of $E'$. In this case we may modify $E'$ as before to obtain $E''$ for which $\gamma_i$ is disjoint from the interior.
By iterating this construction we may assume that $E$ is chosen so that every $\gamma_i$ intersects $D'$ at least once and that $\gamma_i$, $E$ and $D$ does not form a triangle in the interior of $E$.

This `no triangles' condition implies that every arc $\gamma_i$ in the interior of $E$ intersects $D'$ in at least two points. Thus we can conclude that there are precisely two arcs, $\gamma_1$ and $\gamma_2$, and that the arc in the interior of $E$, which we will choose to be $\gamma_1$ intersects $D'$ in two or three points. By the symmetry of the the flyped tongue we may assume that $\gamma_1$ intersect one of the white regions.

Altogether this leaves a highly constrained number of possibilities. It turns out that $\gamma_1$ is essentially determined by the regions where it meets $E$. This leads to the five possibilities which we display in Figures~\ref{fig:gamma1}$(i)$-$(v)$. Note that in Figures~\ref{fig:gamma1}$(ii)$-$(v)$ we are using the primality of $D'$ to allow us to draw $\gamma_1$ as not passing through either of the tangles $T_1$ or $T_2$. Four of these possibilities, the ones in Figures~\ref{fig:gamma1}$(i)$-$(iv)$, can immediately discounted as the exterior of $C$ is reduced. For the remaining possibility, Figure~\ref{fig:gamma1}$(v)$, the arc $\gamma_2$ forms an arc connecting the endpoints of $\gamma_1$ and intersecting $D'$ once. Using again the primality of $D'$ to show tangle contained between $E$ and $\gamma_2$ be unknotted, this must appear as depicted in Figure~\ref{fig:gamma1}$(v)$. As we are assuming that the interior of $C$ contains more than just a single crossing, this possibility can also be ruled out. This completes the proof of the claim.

\begin{figure}[h]
\centering
\def\svgwidth{10cm}
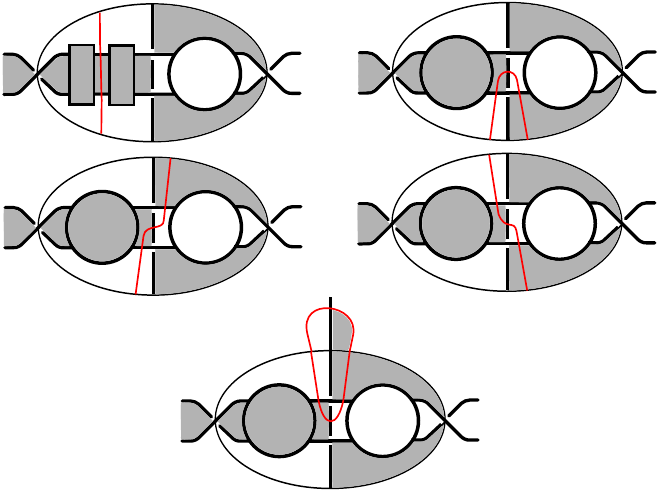
\caption{The five possibilities for $\gamma_1$. Figure~$(v)$ also shows $\gamma_2$.}
\label{fig:gamma1}
\end{figure}
\end{proof}

Assume now that $E$ is disjoint from $C$.
\begin{claim}
If $C$ is contained in the interior of $E$, then it is contained in one of the tangles $T_1$ or $T_2$.
\end{claim}
\begin{proof}[Proof of Claim.]
Suppose that $C$ is contained entirely in the interior of $E$.
If $C$ does not intersect the strand of $D'$ running vertically through the centre of the interior of $E$, then it can clearly be assumed to lie in the interior of $T_1$ or $T_2$. Now consider how $C$ can intersect this vertical strand. Since $C$ must pass through four distinct regions we see that a crossing must lie between any two intersection points. This implies that $C$ can intersect this vertical strand only twice. Furthermore, as $C$ intersects $D'$ in only four points, one of these intersection points must lie between the two crossings with the other lying between one of the crossings and $E$. By considering how these intersection points can be joined up and using the primality of $D'$, this implies that $C$ is a Conway sphere enclosing one of the two crossings on this central strand. In one case, $C$ contains a single crossing in its interior and in the other, the exterior of $C$ is not reduced. Thus we can rule out $C$ intersecting this central strand.
\end{proof}
We now produce the desired isotopy.
\begin{claim}
There is an isotopy, fixing the exterior of $C$, which carries $D'$ to a reduced almost-alternating diagram $D''$ which reduces the number of crossings in the interior of $C$. Moreover, this isotopy can be obtained by flypes followed by an untongue or untwirl move.
\end{claim}
\begin{proof}[Proof of Claim.] By the two preceding claims, we may assume that $E$ is contained in the interior of $C$ or $C$ is contained in one of the tangles $T_1$ or $T_2$. In any case, this means there are flypes fixing the exterior of $C$ to a diagram which admits an untongue move in the interior of $C$. By performing this untongue move we obtain an almost-alternating diagram of the unknot $D''$ with fewer crossings in the interior of $C$. If this new diagram is reduced, then there is nothing further to check. So it remains to consider the ways in which the untongue move could introduce a nugatory crossing in $D''$. Consider the shadings of this diagram before and after the tongue move as shown in Figure~\ref{fig:flypedtonguereduction}. Note that a crossing is nugatory if and only if it is not incident to four distinct shaded regions. As the colourings are unchanged by the untongue move outside of the depicted region, the only crossings which can be nugatory are the two visible in Figure~\ref{fig:flypedtonguereduction}.
Now observe that the dealternator in $D''$ cannot be nugatory. If it were, then the alternating diagram obtained by changing it would also be unknotted. As every crossing is nugatory in an alternating diagram of the unknot \cite{BankwitzNug}, this would imply that every crossing in the exterior of $C$ is nugatory.

Thus, using the labels for regions as shown in Figure~\ref{fig:flypedtonguereduction}, we see the untongue move creates a nugatory crossing only if the the region labelled as $a$ is the same as $c$ or if $b$ is the same as $d$. By the symmetry of the diagram we may assume that $a$ and $c$ are the same. This means $D''$ must appear as in Figure~\ref{fig:containsnugatory} around region $b$. As $D'$ is prime, however, we see that the boxed tangle Figure~\ref{fig:containsnugatory} can contain no crossings. Thus, the nugatory crossing can be removed by Reidemeister~I move. Note that this untongue move followed by a Reidemeister~I move is precisely an untwirl move. Since $C$ must intersect four distinct shaded regions, $C$ must not pass through region $b$. Thus, region $b$ must lie entirely in the interior of $C$. This means that the Reidemeister~I move leaves the exterior of $C$ fixed.
\begin{figure}[h]
\centering
\def\svgwidth{10cm}
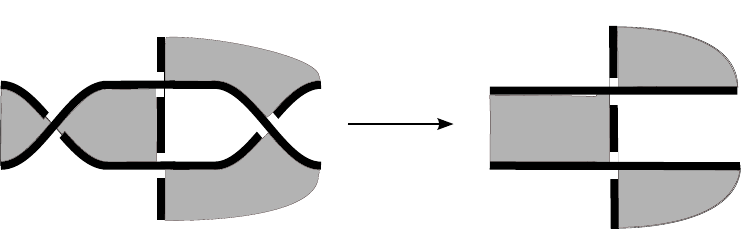
\caption{Checkerboard colourings before and after a tongue move.}
\label{fig:flypedtonguereduction}
\end{figure}
\begin{figure}
  \centering
  \def\svgwidth{4cm}
  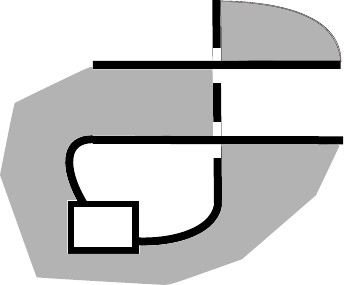
 \caption{The local picture when $D''$ contains a nugatory crossing.}
 \label{fig:containsnugatory}
\end{figure}
\end{proof}
We are now ready to finish the inductive proof. By the last claim there is an isotopy, fixing the exterior of $C$, to a reduced almost-alternating diagram $D''$ with fewer crossings in the interior of $C$. Let $\widetilde{D}$ be the alternating diagram obtained by changing the dealternator of $D''$. Since $C$ has the same exterior in both $D$ and $\widetilde{D}$, we see that $\widetilde{D}$ satisfies the conditions of the lemma. As $\widetilde{D}$ has fewer crossings in the interior of $C$, we may assume by the inductive hypothesis that the interior of $C$ is an alternating unknotting tangle (with distinguished crossing corresponding to the dealternator in $D''$). Since there is an isotopy between the interiors of $C$ in $D'$ and $D''$, this shows that interior of $C$ in $D$ is also an alternating unknotting tangle.
\end{proof}

The final claim in the proof of Lemma~\ref{lem:maxitangles} shows that if the exterior of a Conway sphere in $D$ is reduced then the interior is an alternating unknotting tangle which can be built up from a single crossing by a sequence of flypes, tongue moves and twirl moves.
As any alternating unknotting tangle can be inserted into an alternating diagram so that its exterior is reduced,
this shows that an analogue of Theorem~\ref{thm:altTsuka} holds for alternating unknotting tangles.
\begin{figure}
  \centering
  \def\svgwidth{10cm}
  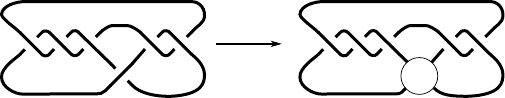
 \caption{Inserting an alternating unknotting tangle to obtain a new alternating diagram with unknotting number one. Note that the exterior of $T$ is reduced.}
 \label{fig:tsuktangleproof}
\end{figure}

\begin{cor}
Any alternating unknotting tangle can be built up from a single crossing by a sequence of flypes, tongue moves and twirl moves.
\end{cor}

\subsection{Substantial Conway spheres}
A satellite knot $K_D$ in $\D$ generated by Construction~\ref{cons:satellite} corresponds to an alternating knot diagram $(D, c)$ containing a Conway sphere. The goal of this subsection is to detect when the converse is true, that is, for what types of Conway spheres in $(D, c)$ the knot $K_D$ is a satellite knot. Recall from Section~\ref{sec:intro} that a Conway sphere $C$ in $(D,c)$ is called substantial if it is visible, the interior of $C$ (i.e. the component of $S^3\setminus C$ containing $c$) contains more than one crossing and the exterior of $C$ is not a rational tangle.

\begin{prop}\label{prop:substantialimpliessatellite}
Let $(D,c)$ be an alternating diagram with an unknotting crossing $c$ containing a substantial Conway sphere. Then the corresponding knot $K_D$ in $\D$ is a satellite knot with companion also in $\D$.
\end{prop}
\begin{proof}
Let $C$ be the substantial Conway sphere in $D$. We may perform a sequence of flype moves on $D$ to obtain a diagram $D'$ that maximizes the number of crossings in the interior of $C$. Since the exterior of $C$ in $D$ is not rational, the exterior of $C$ in $D'$ will also not be rational. In particular, this $D'$ satisfies the hypotheses of Lemma~\ref{lem:maxitangles}. It follows that $D'$ arises by Construction~\ref{cons:satellite}. Consequently, $K_D=K_{D'}$ is a satellite in $\D$.
\end{proof}

\begin{lemma}\label{lem:satelliteimpliesconway}
Let $K_D\in\D$ arise from the diagram $(D, c)$. If $K_D$ is a satellite knot, then $D$ contains a Conway sphere $C$ satisfying either
 \begin{enumerate}[(1)]
 \item $C$ is substantial or
 \item $C$ is hidden and, moreover, the tangle in the exterior of $C$ is not rational.
 \end{enumerate}
\end{lemma}

{\rmk \label{rmk:StrongInversion}By construction, $K_D$ is strongly invertible, that is, there is an involution on the knot exterior $S^3\setminus \nu(K_D)$ that fixes a pair of arcs meeting the boundary torus transversally in four points. 
}
\\

Lemma~\ref{lem:satelliteimpliesconway} is proven by first finding an incompressible, non-boundary parallel torus in the exterior of $K_D$ (not necessarily the one we get from the assumption that $K_D$ is a satellite knot). Given the strong inversion of $S^3\setminus \nu(K_D)$, we quotient the torus to get a Conway sphere $C$. By carefully exploring $C$ we see that the only possibilities are the ones stated in the lemma. For the first part of the argument (i.e. finding the torus), we appeal to the following theorem that we state without proof. The theorem follows directly from \cite[Corollary~4.6]{Holzmann91Equivariant}.

{\thm [Holzmann]\label{thm:Holzmann}
Let $M$ be an orientable, irreducible three-manifold with an involution~$\iota$. Suppose that $M$ contains an incompressible torus, and that $M$ is not an orientable Seifert fiber space over the 2-sphere with four exceptional fibers. Then there is an incompressible torus $R\subset \mathring{M}$, transverse to the fix point set of $\iota$, with either $\iota(R)\cap R = \emptyset$ or $\iota(R) = R$.
}
\begin{proof}[Proof of Lemma~\ref{lem:satelliteimpliesconway}]
Let $\widehat{D}\subset S^3\setminus (\mathring{B^{3}})$ be the diagram obtained by cutting out a small ball containing the unknotting crossing $c$. The double branched cover over $\widehat{D}$ is the knot exterior $S^3\setminus \nu (K_D)$. Since $K_D$ is a satellite knot, $S^3\setminus \nu (K_D)$ contains an incompressible torus $R'$. If $\iota$ is the covering involution on $S^3\setminus \nu (K_D)$ (c.f. Remark~\ref{rmk:StrongInversion}), then, using Theorem~\ref{thm:Holzmann}, we get that $S^3\setminus \nu (K_D)$ contains an incompressible torus $R$, transverse to the fixed set of $\iota$, such that either $\iota(R)=R$ or $\iota(R)\cap R=\emptyset$. Moreover, since we can also apply Theorem~\ref{thm:Holzmann} to any non-integer Dehn filling of $S^3\setminus \nu(K_D)$ and, in particular, all those in which $R'$ remains incompressible, we can assume that $R$ is not boundary parallel. Here we are using that a non-integer surgery on satellite knot is not an orientable Seifert fiber space over the 2-sphere with four exceptional fibers  \cite[Corollary~1.3]{Miyazaki}. An Euler characteristic argument shows that the quotient of $R$ by the involution is either a torus disjoint from $\widehat{D}$ or a Conway sphere (corresponding to either $\iota(R)\cap R=\emptyset$ or $\iota(R)=R$, respectively). We remind the reader that prime alternating knots are not satellites knots \cite[Corollary~1]{Menasco84incompressible}, and in particular, the exterior of $D$ cannot contain an incompressible non-boundary parallel torus. Thus the possibility that the quotient of $R$ be a torus does not occur. Therefore, we get a Conway sphere $C$ in $(D,c)$ which does not intersect the unknotting arc specified by $c$. Since $R$ is incompressible in $S^3\setminus \nu (K_D)$, we see that the exterior of $C$ cannot be a rational tangle. Thus if $C$ is hidden, we see that $(2)$ of the lemma holds. If $C$ is visible, then the interior of $C$ must contain more than one crossing, as otherwise $R$ would have been boundary parallel in $S^3\setminus \nu (K_D)$ contradicting the assumption that $K_D$ is a satellite knot. This shows that if $C$ is visible, it must be substantial.
\end{proof}

\begin{lemma}\label{lem:hiddensphereanalysis}
Suppose that $(D,c)$ contains a hidden Conway sphere disjoint from the unknotting arc specified by $c$ and whose exterior is not rational. Then $D$ contains
a substantial Conway sphere.
\end{lemma}
\begin{proof}
Menasco, in \cite[Theorem~3]{Menasco84incompressible}, shows that if an alternating diagram $D$ contains a hidden Conway sphere, $D$ must appear as in Figure~\ref{fig:hiddenconway}(a). Moreover there is an isotopy of $D$ into a non-alternating knot diagram, as depicted in Figure~\ref{fig:hiddenconway}(b), in which the image of the hidden Conway sphere is visible. See \cite[Figures~3(ii) and 3(iii)]{Thistlethwaite1991}. We may assume that the unknotting crossing $c$ is contained in the tangle $X$. Observe that if any of $Y$, $Z$ or $W$ is not rational, then the visible Conway sphere containing it is a substantial one. Therefore, we may assume that the tangles $Y$, $Z$ and $W$ are all rational.

Now consider Figure~\ref{fig:hiddenconway}(b).
The image of the hidden Conway sphere in the figure is the visible Conway sphere surrounding $X$ and $Y$.
 Since the tangle contained in the exterior of the Conway sphere is not rational, we see that both $Z$ and $W$ contain crossings. It follows from the results of \cite[Section~4]{Kauffman2004}, that in any alternating diagram of a rational tangle at least one pair of arcs emerging from the boundary sphere must meet in a crossing. Thus, if we consider again Figure~\ref{fig:hiddenconway}(a), we see that the exterior of the visible Conway sphere containing $X$ is not rational.  Thus if $X$ contains more than one crossing, the visible Conway sphere surrounding $X$ is substantial. It remains to consider the case where $X$ consists of a single crossing. We will show that in this case $D$ is a diagram of a 2-bridge knot, and argue that this will not happen.

If $X$ only consists of $c$, then the diagram appears as in Figure~\ref{fig:hiddenconway}(c). After changing the unknotting crossing there is an isotopy to the non-prime diagram shown in Figure~\ref{fig:hiddenconway}(d). Since this is the unknot we see that both summands must be unknotted. However, note that the summand containing $Y$ is alternating. Since an alternating diagram of the unknot can contain only nugatory crossings \cite{BankwitzNug} and $D$ is a reduced diagram, this implies that $Y$ does not contain any crossings. Since we are assuming that $Z$ and $W$ are rational, this means that $D$  is a diagram of a 2--bridge knot. This is a contradiction, as any Conway sphere in the exterior of a 2--bridge knot divides the knot diagram into two rational tangles~\cite[Theorem~3.5.18]{Kawauchi2012}.
\end{proof}
\begin{figure}[htp]
\centering
\def\svgheight{\columnheight}
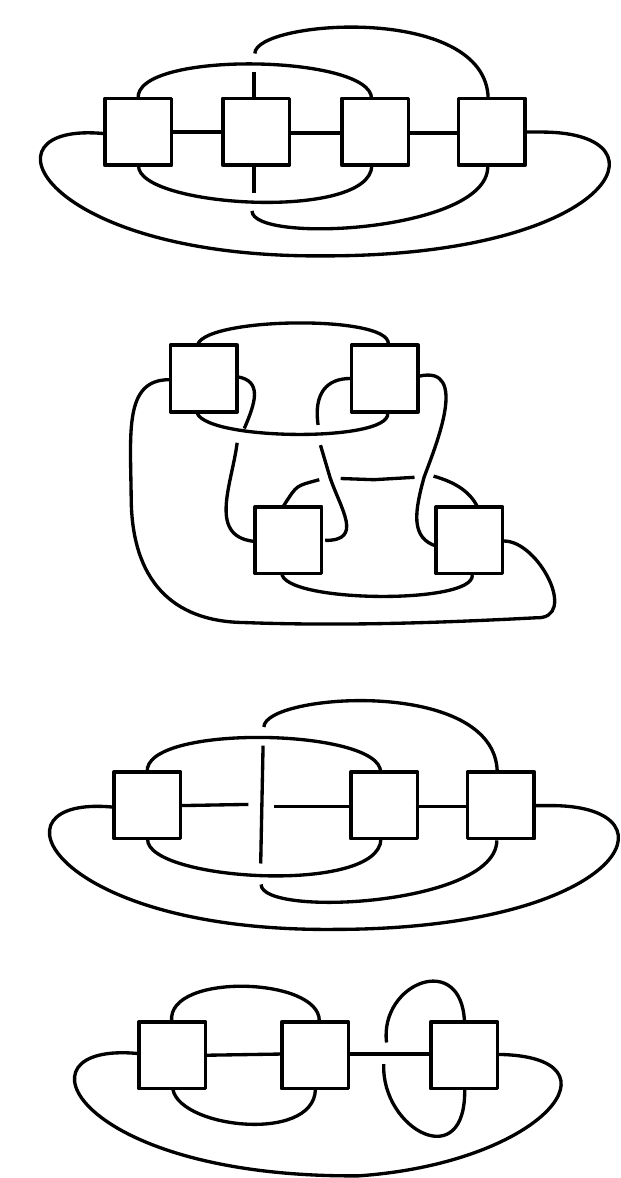

\caption{$(a)$ Shows the standard form of a diagram containing a hidden Conway sphere.
$(b)$ Shows a non-alternating diagram in which the Conway sphere is visible. $(c)$ Shows the diagram when $X$ contains only the unknotting crossing. $(d)$ Shows the non-prime diagram which can be obtained after changing the unknotting crossing in $(c)$.}
\label{fig:hiddenconway}
\end{figure}

\noindent This provides the final step required for our main result on satellite knots in $\D$.
\begin{proof}[Proof of Theorem~\ref{thm:satellitecharacterization}]
$(ii)\Rightarrow (i)$: If $D$ contains a substantial Conway sphere, then Proposition~\ref{prop:substantialimpliessatellite} shows that $K_D$ is a satellite knot.

$(i)\Rightarrow (ii)$: If $K_D$ is a satellite, then Lemma~\ref{lem:satelliteimpliesconway} together with Lemma~\ref{lem:hiddensphereanalysis} imply that $K_D$ contains a substantial Conway sphere.
\end{proof}

\bibliography{Reference}
\bibliographystyle{alpha}

\end{document}

%% file: flypedtongue.pdf_tex
\begingroup%
  \makeatletter%
  \providecommand\color[2][]{%
    \errmessage{(Inkscape) Color is used for the text in Inkscape, but the package 'color.sty' is not loaded}%
    \renewcommand\color[2][]{}%
  }%
  \providecommand\transparent[1]{%
    \errmessage{(Inkscape) Transparency is used (non-zero) for the text in Inkscape, but the package 'transparent.sty' is not loaded}%
    \renewcommand\transparent[1]{}%
  }%
  \providecommand\rotatebox[2]{#2}%
  \ifx\svgwidth\undefined%
    \setlength{\unitlength}{303.3256bp}%
    \ifx\svgscale\undefined%
      \relax%
    \else%
      \setlength{\unitlength}{\unitlength * \real{\svgscale}}%
    \fi%
  \else%
    \setlength{\unitlength}{\svgwidth}%
  \fi%
  \global\let\svgwidth\undefined%
  \global\let\svgscale\undefined%
  \makeatother%
  \begin{picture}(1,0.13574522)%
    \put(0,0){\includegraphics[width=\unitlength]{flypedtongue.pdf}}%
  \end{picture}%
\endgroup%

%% file: claspunknotted2.pdf_tex
\begingroup%
  \makeatletter%
  \providecommand\color[2][]{%
    \errmessage{(Inkscape) Color is used for the text in Inkscape, but the package 'color.sty' is not loaded}%
    \renewcommand\color[2][]{}%
  }%
  \providecommand\transparent[1]{%
    \errmessage{(Inkscape) Transparency is used (non-zero) for the text in Inkscape, but the package 'transparent.sty' is not loaded}%
    \renewcommand\transparent[1]{}%
  }%
  \providecommand\rotatebox[2]{#2}%
  \ifx\svgwidth\undefined%
    \setlength{\unitlength}{65.31302849bp}%
    \ifx\svgscale\undefined%
      \relax%
    \else%
      \setlength{\unitlength}{\unitlength * \real{\svgscale}}%
    \fi%
  \else%
    \setlength{\unitlength}{\svgwidth}%
  \fi%
  \global\let\svgwidth\undefined%
  \global\let\svgscale\undefined%
  \makeatother%
  \begin{picture}(1,0.68472307)%
    \put(0,0){\includegraphics[width=\unitlength]{claspunknotted2.pdf}}%
    \put(0.50932222,0.58223843){\color[rgb]{0,0,0}\makebox(0,0)[lb]{\smash{$\dots$}}}%
  \end{picture}%
\endgroup%

%% file: crossingless.pdf_tex
\begingroup%
  \makeatletter%
  \providecommand\color[2][]{%
    \errmessage{(Inkscape) Color is used for the text in Inkscape, but the package 'color.sty' is not loaded}%
    \renewcommand\color[2][]{}%
  }%
  \providecommand\transparent[1]{%
    \errmessage{(Inkscape) Transparency is used (non-zero) for the text in Inkscape, but the package 'transparent.sty' is not loaded}%
    \renewcommand\transparent[1]{}%
  }%
  \providecommand\rotatebox[2]{#2}%
  \ifx\svgwidth\undefined%
    \setlength{\unitlength}{75.05983234bp}%
    \ifx\svgscale\undefined%
      \relax%
    \else%
      \setlength{\unitlength}{\unitlength * \real{\svgscale}}%
    \fi%
  \else%
    \setlength{\unitlength}{\svgwidth}%
  \fi%
  \global\let\svgwidth\undefined%
  \global\let\svgscale\undefined%
  \makeatother%
  \begin{picture}(1,0.99842821)%
    \put(0,0){\includegraphics[width=\unitlength,page=1]{crossingless.pdf}}%
  \end{picture}%
\endgroup%

%% file: basicflype.pdf_tex
\begingroup%
  \makeatletter%
  \providecommand\color[2][]{%
    \errmessage{(Inkscape) Color is used for the text in Inkscape, but the package 'color.sty' is not loaded}%
    \renewcommand\color[2][]{}%
  }%
  \providecommand\transparent[1]{%
    \errmessage{(Inkscape) Transparency is used (non-zero) for the text in Inkscape, but the package 'transparent.sty' is not loaded}%
    \renewcommand\transparent[1]{}%
  }%
  \providecommand\rotatebox[2]{#2}%
  \ifx\svgwidth\undefined%
    \setlength{\unitlength}{1493.15bp}%
    \ifx\svgscale\undefined%
      \relax%
    \else%
      \setlength{\unitlength}{\unitlength * \real{\svgscale}}%
    \fi%
  \else%
    \setlength{\unitlength}{\svgwidth}%
  \fi%
  \global\let\svgwidth\undefined%
  \global\let\svgscale\undefined%
  \makeatother%
  \begin{picture}(1,0.16815122)%
    \put(0,0){\includegraphics[width=\unitlength]{basicflype.pdf}}%
  \end{picture}%
\endgroup%

%% file: alttongueandswirl.pdf_tex
\begingroup%
  \makeatletter%
  \providecommand\color[2][]{%
    \errmessage{(Inkscape) Color is used for the text in Inkscape, but the package 'color.sty' is not loaded}%
    \renewcommand\color[2][]{}%
  }%
  \providecommand\transparent[1]{%
    \errmessage{(Inkscape) Transparency is used (non-zero) for the text in Inkscape, but the package 'transparent.sty' is not loaded}%
    \renewcommand\transparent[1]{}%
  }%
  \providecommand\rotatebox[2]{#2}%
  \ifx\svgwidth\undefined%
    \setlength{\unitlength}{345.35104bp}%
    \ifx\svgscale\undefined%
      \relax%
    \else%
      \setlength{\unitlength}{\unitlength * \real{\svgscale}}%
    \fi%
  \else%
    \setlength{\unitlength}{\svgwidth}%
  \fi%
  \global\let\svgwidth\undefined%
  \global\let\svgscale\undefined%
  \makeatother%
  \begin{picture}(1,0.7115224)%
    \put(0,0){\includegraphics[width=\unitlength]{alttongueandswirl.pdf}}%
  \end{picture}%
\endgroup%

%% file: claspknotted2.pdf_tex
\begingroup%
  \makeatletter%
  \providecommand\color[2][]{%
    \errmessage{(Inkscape) Color is used for the text in Inkscape, but the package 'color.sty' is not loaded}%
    \renewcommand\color[2][]{}%
  }%
  \providecommand\transparent[1]{%
    \errmessage{(Inkscape) Transparency is used (non-zero) for the text in Inkscape, but the package 'transparent.sty' is not loaded}%
    \renewcommand\transparent[1]{}%
  }%
  \providecommand\rotatebox[2]{#2}%
  \ifx\svgwidth\undefined%
    \setlength{\unitlength}{169.86302849bp}%
    \ifx\svgscale\undefined%
      \relax%
    \else%
      \setlength{\unitlength}{\unitlength * \real{\svgscale}}%
    \fi%
  \else%
    \setlength{\unitlength}{\svgwidth}%
  \fi%
  \global\let\svgwidth\undefined%
  \global\let\svgscale\undefined%
  \makeatother%
  \begin{picture}(1,0.27093204)%
    \put(0,0){\includegraphics[width=\unitlength]{claspknotted2.pdf}}%
    \put(0.19583647,0.23152628){\color[rgb]{0,0,0}\makebox(0,0)[lb]{\smash{$\dots$}}}%
    \put(0.81131083,0.22383438){\color[rgb]{0,0,0}\makebox(0,0)[lb]{\smash{$\dots$}}}%
  \end{picture}%
\endgroup%

%% file: twirl.pdf_tex
\begingroup%
  \makeatletter%
  \providecommand\color[2][]{%
    \errmessage{(Inkscape) Color is used for the text in Inkscape, but the package 'color.sty' is not loaded}%
    \renewcommand\color[2][]{}%
  }%
  \providecommand\transparent[1]{%
    \errmessage{(Inkscape) Transparency is used (non-zero) for the text in Inkscape, but the package 'transparent.sty' is not loaded}%
    \renewcommand\transparent[1]{}%
  }%
  \providecommand\rotatebox[2]{#2}%
  \ifx\svgwidth\undefined%
    \setlength{\unitlength}{1447.373258bp}%
    \ifx\svgscale\undefined%
      \relax%
    \else%
      \setlength{\unitlength}{\unitlength * \real{\svgscale}}%
    \fi%
  \else%
    \setlength{\unitlength}{\svgwidth}%
  \fi%
  \global\let\svgwidth\undefined%
  \global\let\svgscale\undefined%
  \makeatother%
  \begin{picture}(1,0.44818887)%
    \put(0,0){\includegraphics[width=\unitlength]{twirl.pdf}}%
  \end{picture}%
\endgroup%

%% file: bandsum2.pdf_tex
\begingroup%
  \makeatletter%
  \providecommand\color[2][]{%
    \errmessage{(Inkscape) Color is used for the text in Inkscape, but the package 'color.sty' is not loaded}%
    \renewcommand\color[2][]{}%
  }%
  \providecommand\transparent[1]{%
    \errmessage{(Inkscape) Transparency is used (non-zero) for the text in Inkscape, but the package 'transparent.sty' is not loaded}%
    \renewcommand\transparent[1]{}%
  }%
  \providecommand\rotatebox[2]{#2}%
  \ifx\svgwidth\undefined%
    \setlength{\unitlength}{341.650592bp}%
    \ifx\svgscale\undefined%
      \relax%
    \else%
      \setlength{\unitlength}{\unitlength * \real{\svgscale}}%
    \fi%
  \else%
    \setlength{\unitlength}{\svgwidth}%
  \fi%
  \global\let\svgwidth\undefined%
  \global\let\svgscale\undefined%
  \makeatother%
  \begin{picture}(1,0.34428449)%
    \put(0,0){\includegraphics[width=\unitlength]{bandsum2.pdf}}%
  \end{picture}%
\endgroup%

%% file: 2-bridge.pdf_tex
\begingroup%
  \makeatletter%
  \providecommand\color[2][]{%
    \errmessage{(Inkscape) Color is used for the text in Inkscape, but the package 'color.sty' is not loaded}%
    \renewcommand\color[2][]{}%
  }%
  \providecommand\transparent[1]{%
    \errmessage{(Inkscape) Transparency is used (non-zero) for the text in Inkscape, but the package 'transparent.sty' is not loaded}%
    \renewcommand\transparent[1]{}%
  }%
  \providecommand\rotatebox[2]{#2}%
  \ifx\svgwidth\undefined%
    \setlength{\unitlength}{122.40237602bp}%
    \ifx\svgscale\undefined%
      \relax%
    \else%
      \setlength{\unitlength}{\unitlength * \real{\svgscale}}%
    \fi%
  \else%
    \setlength{\unitlength}{\svgwidth}%
  \fi%
  \global\let\svgwidth\undefined%
  \global\let\svgscale\undefined%
  \makeatother%
  \begin{picture}(1,1.02384923)%
    \put(0,0){\includegraphics[width=\unitlength,page=1]{2-bridge.pdf}}%
    \put(0.20293704,0.75349842){\color[rgb]{0,0,0}\makebox(0,0)[lb]{\smash{{\large $A$}}}}%
    \put(0.71496668,0.24080247){\color[rgb]{0,0,0}\makebox(0,0)[lb]{\smash{{\large $B$}}}}%
    \put(0,0){\includegraphics[width=\unitlength,page=2]{2-bridge.pdf}}%
  \end{picture}%
\endgroup%

%% file: unknottingtangle.pdf_tex
\begingroup%
  \makeatletter%
  \providecommand\color[2][]{%
    \errmessage{(Inkscape) Color is used for the text in Inkscape, but the package 'color.sty' is not loaded}%
    \renewcommand\color[2][]{}%
  }%
  \providecommand\transparent[1]{%
    \errmessage{(Inkscape) Transparency is used (non-zero) for the text in Inkscape, but the package 'transparent.sty' is not loaded}%
    \renewcommand\transparent[1]{}%
  }%
  \providecommand\rotatebox[2]{#2}%
  \ifx\svgwidth\undefined%
    \setlength{\unitlength}{171.44935533bp}%
    \ifx\svgscale\undefined%
      \relax%
    \else%
      \setlength{\unitlength}{\unitlength * \real{\svgscale}}%
    \fi%
  \else%
    \setlength{\unitlength}{\svgwidth}%
  \fi%
  \global\let\svgwidth\undefined%
  \global\let\svgscale\undefined%
  \makeatother%
  \begin{picture}(1,0.32958449)%
    \put(0,0){\includegraphics[width=\unitlength]{unknottingtangle.pdf}}%
  \end{picture}%
\endgroup%

%% file: maximaldiagram.pdf_tex
\begingroup%
  \makeatletter%
  \providecommand\color[2][]{%
    \errmessage{(Inkscape) Color is used for the text in Inkscape, but the package 'color.sty' is not loaded}%
    \renewcommand\color[2][]{}%
  }%
  \providecommand\transparent[1]{%
    \errmessage{(Inkscape) Transparency is used (non-zero) for the text in Inkscape, but the package 'transparent.sty' is not loaded}%
    \renewcommand\transparent[1]{}%
  }%
  \providecommand\rotatebox[2]{#2}%
  \ifx\svgwidth\undefined%
    \setlength{\unitlength}{153.72526114bp}%
    \ifx\svgscale\undefined%
      \relax%
    \else%
      \setlength{\unitlength}{\unitlength * \real{\svgscale}}%
    \fi%
  \else%
    \setlength{\unitlength}{\svgwidth}%
  \fi%
  \global\let\svgwidth\undefined%
  \global\let\svgscale\undefined%
  \makeatother%
  \begin{picture}(1,0.79459868)%
    \put(0,0){\includegraphics[width=\unitlength]{maximaldiagram.pdf}}%
    \put(0.83332371,0.50497859){\color[rgb]{0,0,0}\makebox(0,0)[lb]{\smash{$T_2$}}}%
    \put(0.607775,0.00536164){\color[rgb]{1,0,0}\makebox(0,0)[lb]{\smash{$C$}}}%
    \put(0.46010214,0.22329995){\color[rgb]{0,0,0}\makebox(0,0)[lb]{\smash{$T_3$}}}%
    \put(0.08605853,0.50149925){\color[rgb]{0,0,0}\makebox(0,0)[lb]{\smash{$T_1$}}}%
  \end{picture}%
\endgroup%

%% file: detailedtongue.pdf_tex
\begingroup%
  \makeatletter%
  \providecommand\color[2][]{%
    \errmessage{(Inkscape) Color is used for the text in Inkscape, but the package 'color.sty' is not loaded}%
    \renewcommand\color[2][]{}%
  }%
  \providecommand\transparent[1]{%
    \errmessage{(Inkscape) Transparency is used (non-zero) for the text in Inkscape, but the package 'transparent.sty' is not loaded}%
    \renewcommand\transparent[1]{}%
  }%
  \providecommand\rotatebox[2]{#2}%
  \ifx\svgwidth\undefined%
    \setlength{\unitlength}{143.67399163bp}%
    \ifx\svgscale\undefined%
      \relax%
    \else%
      \setlength{\unitlength}{\unitlength * \real{\svgscale}}%
    \fi%
  \else%
    \setlength{\unitlength}{\svgwidth}%
  \fi%
  \global\let\svgwidth\undefined%
  \global\let\svgscale\undefined%
  \makeatother%
  \begin{picture}(1,0.40397832)%
    \put(0,0){\includegraphics[width=\unitlength]{detailedtongue.pdf}}%
    \put(0.62052809,0.1702802){\color[rgb]{0,0,0}\makebox(0,0)[lb]{\smash{$T_2$}}}%
    \put(0.078514,0.34461703){\color[rgb]{0,0,0}\makebox(0,0)[lb]{\smash{$E$}}}%
    \put(0.29293325,0.18201832){\color[rgb]{0,0,0}\makebox(0,0)[lb]{\smash{$T_1$}}}%
  \end{picture}%
\endgroup%

%% file: fourcolours.pdf_tex
\begingroup%
  \makeatletter%
  \providecommand\color[2][]{%
    \errmessage{(Inkscape) Color is used for the text in Inkscape, but the package 'color.sty' is not loaded}%
    \renewcommand\color[2][]{}%
  }%
  \providecommand\transparent[1]{%
    \errmessage{(Inkscape) Transparency is used (non-zero) for the text in Inkscape, but the package 'transparent.sty' is not loaded}%
    \renewcommand\transparent[1]{}%
  }%
  \providecommand\rotatebox[2]{#2}%
  \ifx\svgwidth\undefined%
    \setlength{\unitlength}{80.06056637bp}%
    \ifx\svgscale\undefined%
      \relax%
    \else%
      \setlength{\unitlength}{\unitlength * \real{\svgscale}}%
    \fi%
  \else%
    \setlength{\unitlength}{\svgwidth}%
  \fi%
  \global\let\svgwidth\undefined%
  \global\let\svgscale\undefined%
  \makeatother%
  \begin{picture}(1,0.99844453)%
    \put(0,0){\includegraphics[width=\unitlength]{fourcolours.pdf}}%
  \end{picture}%
\endgroup%

%% file: ModifyingE.pdf_tex
\begingroup%
  \makeatletter%
  \providecommand\color[2][]{%
    \errmessage{(Inkscape) Color is used for the text in Inkscape, but the package 'color.sty' is not loaded}%
    \renewcommand\color[2][]{}%
  }%
  \providecommand\transparent[1]{%
    \errmessage{(Inkscape) Transparency is used (non-zero) for the text in Inkscape, but the package 'transparent.sty' is not loaded}%
    \renewcommand\transparent[1]{}%
  }%
  \providecommand\rotatebox[2]{#2}%
  \ifx\svgwidth\undefined%
    \setlength{\unitlength}{435.225bp}%
    \ifx\svgscale\undefined%
      \relax%
    \else%
      \setlength{\unitlength}{\unitlength * \real{\svgscale}}%
    \fi%
  \else%
    \setlength{\unitlength}{\svgwidth}%
  \fi%
  \global\let\svgwidth\undefined%
  \global\let\svgscale\undefined%
  \makeatother%
  \begin{picture}(1,0.51597422)%
    \put(0,0){\includegraphics[width=\unitlength]{ModifyingE.pdf}}%
    \put(0.14295562,0.48915616){\color[rgb]{0,0,0}\makebox(0,0)[lb]{\smash{$E$}}}%
    \put(0.65455918,0.33909973){\color[rgb]{0,0,0}\makebox(0,0)[lb]{\smash{$E'$}}}%
    \put(0.15739805,0.36705183){\color[rgb]{0,0,0}\makebox(0,0)[lb]{\smash{$\gamma_i$}}}%
    \put(0.18879566,0.22988165){\color[rgb]{0,0,0}\makebox(0,0)[lb]{\smash{$D'$}}}%
    \put(0.03932396,0.10269144){\color[rgb]{0,0,0}\makebox(0,0)[lb]{\smash{$E$}}}%
    \put(0.17486973,0.06462721){\color[rgb]{0,0,0}\makebox(0,0)[lb]{\smash{$\gamma_i$}}}%
    \put(0.55851112,0.106405){\color[rgb]{0,0,0}\makebox(0,0)[lb]{\smash{$E'$}}}%
  \end{picture}%
\endgroup%

%% file: gamma1options.pdf_tex
\begingroup%
  \makeatletter%
  \providecommand\color[2][]{%
    \errmessage{(Inkscape) Color is used for the text in Inkscape, but the package 'color.sty' is not loaded}%
    \renewcommand\color[2][]{}%
  }%
  \providecommand\transparent[1]{%
    \errmessage{(Inkscape) Transparency is used (non-zero) for the text in Inkscape, but the package 'transparent.sty' is not loaded}%
    \renewcommand\transparent[1]{}%
  }%
  \providecommand\rotatebox[2]{#2}%
  \ifx\svgwidth\undefined%
    \setlength{\unitlength}{315.3357104bp}%
    \ifx\svgscale\undefined%
      \relax%
    \else%
      \setlength{\unitlength}{\unitlength * \real{\svgscale}}%
    \fi%
  \else%
    \setlength{\unitlength}{\svgwidth}%
  \fi%
  \global\let\svgwidth\undefined%
  \global\let\svgscale\undefined%
  \makeatother%
  \begin{picture}(1,0.7446744)%
    \put(0,0){\includegraphics[width=\unitlength]{gamma1options.pdf}}%
    \put(-0.00210589,0.71863197){\color[rgb]{0,0,0}\makebox(0,0)[lb]{\smash{$(i)$}}}%
    \put(0.54509398,0.7253993){\color[rgb]{0,0,0}\makebox(0,0)[lb]{\smash{$(ii)$}}}%
    \put(0.00103498,0.4745748){\color[rgb]{0,0,0}\makebox(0,0)[lb]{\smash{$(iii)$}}}%
    \put(0.54388216,0.48790365){\color[rgb]{0,0,0}\makebox(0,0)[lb]{\smash{$(iv)$}}}%
    \put(0.27367031,0.19345749){\color[rgb]{0,0,0}\makebox(0,0)[lb]{\smash{$(v)$}}}%
  \end{picture}%
\endgroup%

%% file: shadedtongue.pdf_tex
\begingroup%
  \makeatletter%
  \providecommand\color[2][]{%
    \errmessage{(Inkscape) Color is used for the text in Inkscape, but the package 'color.sty' is not loaded}%
    \renewcommand\color[2][]{}%
  }%
  \providecommand\transparent[1]{%
    \errmessage{(Inkscape) Transparency is used (non-zero) for the text in Inkscape, but the package 'transparent.sty' is not loaded}%
    \renewcommand\transparent[1]{}%
  }%
  \providecommand\rotatebox[2]{#2}%
  \ifx\svgwidth\undefined%
    \setlength{\unitlength}{355.675bp}%
    \ifx\svgscale\undefined%
      \relax%
    \else%
      \setlength{\unitlength}{\unitlength * \real{\svgscale}}%
    \fi%
  \else%
    \setlength{\unitlength}{\svgwidth}%
  \fi%
  \global\let\svgwidth\undefined%
  \global\let\svgscale\undefined%
  \makeatother%
  \begin{picture}(1,0.33070921)%
    \put(0,0){\includegraphics[width=\unitlength]{shadedtongue.pdf}}%
    \put(0.67857317,0.14838532){\color[rgb]{0,0,0}\makebox(0,0)[lb]{\smash{$a$}}}%
    \put(0.75247686,0.06002215){\color[rgb]{0,0,0}\makebox(0,0)[lb]{\smash{$b$}}}%
    \put(0.91635039,0.1499919){\color[rgb]{0,0,0}\makebox(0,0)[lb]{\smash{$d$}}}%
    \put(0.87779192,0.05841552){\color[rgb]{0,0,0}\makebox(0,0)[lb]{\smash{$c$}}}%
  \end{picture}%
\endgroup%

%% file: containsnugatory.pdf_tex
\begingroup%
  \makeatletter%
  \providecommand\color[2][]{%
    \errmessage{(Inkscape) Color is used for the text in Inkscape, but the package 'color.sty' is not loaded}%
    \renewcommand\color[2][]{}%
  }%
  \providecommand\transparent[1]{%
    \errmessage{(Inkscape) Transparency is used (non-zero) for the text in Inkscape, but the package 'transparent.sty' is not loaded}%
    \renewcommand\transparent[1]{}%
  }%
  \providecommand\rotatebox[2]{#2}%
  \ifx\svgwidth\undefined%
    \setlength{\unitlength}{164.7242928bp}%
    \ifx\svgscale\undefined%
      \relax%
    \else%
      \setlength{\unitlength}{\unitlength * \real{\svgscale}}%
    \fi%
  \else%
    \setlength{\unitlength}{\svgwidth}%
  \fi%
  \global\let\svgwidth\undefined%
  \global\let\svgscale\undefined%
  \makeatother%
  \begin{picture}(1,0.83006856)%
    \put(0,0){\includegraphics[width=\unitlength]{containsnugatory.pdf}}%
    \put(0.46853926,0.32088559){\color[rgb]{0,0,0}\makebox(0,0)[lb]{\smash{$b$}}}%
    \put(0.82237728,0.51514959){\color[rgb]{0,0,0}\makebox(0,0)[lb]{\smash{$d$}}}%
    \put(0.32883553,0.50824663){\color[rgb]{0,0,0}\makebox(0,0)[lb]{\smash{$a=c$}}}%
  \end{picture}%
\endgroup%

%% file: tsuktangle.pdf_tex
\begingroup%
  \makeatletter%
  \providecommand\color[2][]{%
    \errmessage{(Inkscape) Color is used for the text in Inkscape, but the package 'color.sty' is not loaded}%
    \renewcommand\color[2][]{}%
  }%
  \providecommand\transparent[1]{%
    \errmessage{(Inkscape) Transparency is used (non-zero) for the text in Inkscape, but the package 'transparent.sty' is not loaded}%
    \renewcommand\transparent[1]{}%
  }%
  \providecommand\rotatebox[2]{#2}%
  \ifx\svgwidth\undefined%
    \setlength{\unitlength}{242.13633541bp}%
    \ifx\svgscale\undefined%
      \relax%
    \else%
      \setlength{\unitlength}{\unitlength * \real{\svgscale}}%
    \fi%
  \else%
    \setlength{\unitlength}{\svgwidth}%
  \fi%
  \global\let\svgwidth\undefined%
  \global\let\svgscale\undefined%
  \makeatother%
  \begin{picture}(1,0.19403712)%
    \put(0,0){\includegraphics[width=\unitlength]{tsuktangle.pdf}}%
    \put(0.81806481,0.02938899){\color[rgb]{0,0,0}\makebox(0,0)[lb]{\smash{$T$}}}%
  \end{picture}%
\endgroup%

%% file: hiddenconway.pdf_tex
\begingroup%
  \makeatletter%
  \providecommand\color[2][]{%
    \errmessage{(Inkscape) Color is used for the text in Inkscape, but the package 'color.sty' is not loaded}%
    \renewcommand\color[2][]{}%
  }%
  \providecommand\transparent[1]{%
    \errmessage{(Inkscape) Transparency is used (non-zero) for the text in Inkscape, but the package 'transparent.sty' is not loaded}%
    \renewcommand\transparent[1]{}%
  }%
  \providecommand\rotatebox[2]{#2}%
  \ifx\svgwidth\undefined%
    \setlength{\unitlength}{297.63983766bp}%
    \ifx\svgscale\undefined%
      \relax%
    \else%
      \setlength{\unitlength}{\unitlength * \real{\svgscale}}%
    \fi%
  \else%
    \setlength{\unitlength}{\svgwidth}%
  \fi%
  \global\let\svgwidth\undefined%
  \global\let\svgscale\undefined%
  \makeatother%
  \begin{picture}(1,1.89901245)%
    \put(0,0){\includegraphics[width=\unitlength]{hiddenconway.pdf}}%
    \put(0.20567592,1.67016819){\color[rgb]{0,0,0}\makebox(0,0)[lb]{\smash{{\large $Z$}}}}%
    \put(0.39570423,1.67016819){\color[rgb]{0,0,0}\makebox(0,0)[lb]{\smash{{\large $X$}}}}%
    \put(0.58761401,1.67016819){\color[rgb]{0,0,0}\makebox(0,0)[lb]{\smash{{\large $W$}}}}%
    \put(0.77602967,1.67016819){\color[rgb]{0,0,0}\makebox(0,0)[lb]{\smash{{\large $Y$}}}}%
    \put(0.60320337,1.27237198){\color[rgb]{0,0,0}\makebox(0,0)[lb]{\smash{{\large $W$}}}}%
    \put(-0.00095806,1.87800079){\color[rgb]{0,0,0}\makebox(0,0)[lb]{\smash{$(a)$}}}%
    \put(-0.00095806,1.3538793){\color[rgb]{0,0,0}\makebox(0,0)[lb]{\smash{$(b)$}}}%
    \put(0.22019009,0.58429204){\color[rgb]{0,0,0}\makebox(0,0)[lb]{\smash{{\large $Z$}}}}%
    \put(0.60239698,0.58429204){\color[rgb]{0,0,0}\makebox(0,0)[lb]{\smash{{\large $W$}}}}%
    \put(0.79081263,0.58429204){\color[rgb]{0,0,0}\makebox(0,0)[lb]{\smash{{\large $Y$}}}}%
    \put(-0.00095806,0.79181866){\color[rgb]{0,0,0}\makebox(0,0)[lb]{\smash{$(c)$}}}%
    \put(-0.00095806,0.31377204){\color[rgb]{0,0,0}\makebox(0,0)[lb]{\smash{$(d)$}}}%
    \put(0.44784777,1.01165419){\color[rgb]{0,0,0}\makebox(0,0)[lb]{\smash{{\large $X$}}}}%
    \put(0.7397442,1.01165419){\color[rgb]{0,0,0}\makebox(0,0)[lb]{\smash{{\large $Y$}}}}%
    \put(0.3115757,1.27237198){\color[rgb]{0,0,0}\makebox(0,0)[lb]{\smash{{\large $Z$}}}}%
    \put(0.26064167,0.18112021){\color[rgb]{0,0,0}\makebox(0,0)[lb]{\smash{{\large $Z$}}}}%
    \put(0.49112157,0.18112021){\color[rgb]{0,0,0}\makebox(0,0)[lb]{\smash{{\large $W$}}}}%
    \put(0.7312238,0.18112021){\color[rgb]{0,0,0}\makebox(0,0)[lb]{\smash{{\large $Y$}}}}%
  \end{picture}%
\endgroup%

%% file: DBC-NYJ-revised.bbl
\begin{thebibliography}{CGLS87}

\bibitem[Ban30]{BankwitzNug}
Carl Bankwitz.
\newblock {\"U}ber die {T}orsionszahlen der alternierenden {K}noten.
\newblock {\em Math. Ann.}, 103(1):145--161, 1930.

\bibitem[Ber18]{Berge}
John Berge.
\newblock {Some knots with surgeries yielding lens spaces}.
\newblock {\em preprint}, page arXiv:1802.09722, 2018.

\bibitem[BM15]{Baker2015}
Kenneth Baker and Kimihiko Motegi.
\newblock {Twist families of L-space knots, their genera, and Seifert
  surgeries}.
\newblock {\em preprint}, page arXiv:1506.04455, 2015.

\bibitem[CGLS87]{Culler1987CyclicSurgery}
Marc Culler, Cameron Gordon, John Luecke, and Peter Shalen.
\newblock Dehn surgery on knots.
\newblock {\em Ann. of Math.}, 125(2):237--300, 1987.

\bibitem[Ghi08]{Ghiggini}
Paolo Ghiggini.
\newblock Knot {F}loer homology detects genus-one fibred knots.
\newblock {\em Amer. J. Math.}, 130(5):1151--1169, 2008.

\bibitem[GL04]{Gordon2004non}
Cameron~McA. Gordon and John Luecke.
\newblock {Non-integral toroidal Dehn surgeries}.
\newblock {\em Comm. Anal. Geom.}, 12:417--485, 2004.

\bibitem[GL06]{Gordon2006knots}
Cameron~McA Gordon and John Luecke.
\newblock Knots with unknotting number 1 and essential conway spheres.
\newblock {\em Algebr. Geom. Topol.}, 6(5):2051--2116, 2006.

\bibitem[Hed10]{Hedden2010}
Matthew Hedden.
\newblock Notions of positivity and the {O}zsv\'ath-{S}zab\'o concordance
  invariant.
\newblock {\em J. Knot Theory Ramifications}, 19(5):617--629, 2010.

\bibitem[Hed11]{HeddenBerge}
Matthew Hedden.
\newblock On {F}loer homology and the {B}erge conjecture on knots admitting
  lens space surgeries.
\newblock {\em Trans. Amer. Math. Soc.}, 363(2):949--968, 2011.

\bibitem[HLV14]{HomLidmanVafaee}
Jennifer Hom, Tye Lidman, and Faramarz Vafaee.
\newblock Berge-{G}abai knots and {L}-space satellite operations.
\newblock {\em Algebr. Geom. Topol.}, 14(6):3745--3763, 2014.

\bibitem[Hol91]{Holzmann91Equivariant}
W.~H. Holzmann.
\newblock An equivariant torus theorem for involutions.
\newblock {\em Trans. Amer. Math. Soc.}, 326(2):887--906, 1991.

\bibitem[Hom11]{Hom2011a}
Jennifer Hom.
\newblock A note on cabling and {$L$}-space surgeries.
\newblock {\em Algebr. Geom. Topol.}, 11(1):219--223, 2011.

\bibitem[Hom16]{Hom2016}
Jennifer Hom.
\newblock {Satellite knots and L-space surgeries}.
\newblock {\em Bull. London Math. Soc.}, 48(5):771--778, 2016.

\bibitem[HR85]{Hodgson1985}
Craig Hodgson and J.~Hyam Rubinstein.
\newblock Involutions and isotopies of lens spaces.
\newblock In {\em Knot theory and manifolds ({V}ancouver, {B}.{C}., 1983),
  volume 1144 of {L}ecture {N}otes in {M}ath.}, pages 60--96. Springer, Berlin,
  1985.

\bibitem[HW18]{Hedden2014}
Matthew Hedden and Liam Watson.
\newblock {On the geography and botany of knot Floer homology}.
\newblock {\em Sel. Math.}, 24(2):997--1037, 2018.

\bibitem[Kaw12]{Kawauchi2012}
Akio Kawauchi.
\newblock {\em A survey of knot theory}.
\newblock Birkh{\"a}user Verlag, Basel, 2012.

\bibitem[KL04]{Kauffman2004}
Louis~H. Kauffman and Sofia Lambropoulou.
\newblock On the classification of rational tangles.
\newblock {\em Adv. in Appl. Math.}, 33(2):199--237, 2004.

\bibitem[KM86]{Murakami1986}
Taizo Kanenobu and Hitoshi Murakami.
\newblock Two-bridge knots with unknotting number one.
\newblock {\em Proc. Amer. Math. Soc.}, 98(3):499--502, 1986.

\bibitem[Krc18]{Krcatovich2014}
David Krcatovich.
\newblock {A restriction on the Alexander polynomials of L-space knots}.
\newblock {\em Pac. J. Math.}, 297(1):117--129, 2018.

\bibitem[McC17a]{mccoy17unknotting}
Duncan McCoy.
\newblock Alternating knots with unknotting number one.
\newblock {\em Adv. Math.}, 305:757--802, 2017.

\bibitem[McC17b]{mccoy2014bounds}
Duncan McCoy.
\newblock Bounds on alternating surgery slopes.
\newblock {\em Algebr. Geom. Topol.}, 17(5):2603--2634, 2017.

\bibitem[Men84]{Menasco84incompressible}
William Menasco.
\newblock Closed incompressible surfaces in alternating knot and link
  complements.
\newblock {\em Topology}, 23(1):37--44, 1984.

\bibitem[MM97]{Miyazaki}
Katura Miyazaki and Kimihiko Motegi.
\newblock Seifert fibred manifolds and {D}ehn surgery.
\newblock {\em Topology}, 36(2):579--603, 1997.

\bibitem[Mon73]{montesinos1973}
Jose Montesinos.
\newblock Seifert manifolds that are ramified two-sheeted cyclic coverings.
\newblock {\em Bol. Soc. Mat. Mexicana (2)}, 18:1--32, 1973.

\bibitem[Mos71]{Moser71elementary}
Louise Moser.
\newblock Elementary surgery along a torus knot.
\newblock {\em Pacific J. Math.}, 38:737--745, 1971.

\bibitem[Mot16]{Motegi2014}
Kimihiko Motegi.
\newblock {L-space surgery and twisting operation}.
\newblock {\em Algebr. Geom. Topol.}, 16(3):1727--1772, 2016.

\bibitem[MT14]{Motegi2014b}
Kimihiko Motegi and Kazushige Tohki.
\newblock Hyperbolic {L}-space knots and exceptional {D}ehn surgeries.
\newblock {\em J. Knot Theory Ramifications}, 23(14), 2014.

\bibitem[Ni07]{Ni2009}
Yi~Ni.
\newblock Knot {F}loer homology detects fibred knots.
\newblock {\em Invent. Math.}, 170(3):577--608, 2007.

\bibitem[OS05a]{Ozsvath2005a}
Peter Ozsv{\'a}th and Zolt{\'a}n Szab{\'o}.
\newblock Knots with unknotting number one and {H}eegaard {F}loer homology.
\newblock {\em Topology}, 44(4):705--745, 2005.

\bibitem[OS05b]{Ath}
Peter Ozsv{\'a}th and Zolt{\'a}n Szab{\'o}.
\newblock On knot {F}loer homology and lens space surgeries.
\newblock {\em Topology}, 44(6):1281--1300, 2005.

\bibitem[OS05c]{OS-BDC2005}
Peter Ozsv{\'a}th and Zolt{\'a}n Szab{\'o}.
\newblock On the {H}eegaard {F}loer homology of branched double-covers.
\newblock {\em Adv. Math.}, 194(1):1--33, 2005.

\bibitem[Thi91]{Thistlethwaite1991}
Morwen Thistlethwaite.
\newblock On the algebraic part of an alternating link.
\newblock {\em Pacific J. Math.}, 151(2):317--333, 1991.

\bibitem[Thu82]{Thurston1982}
William Thurston.
\newblock {Three dimensional manifolds, Kleinian groups and hyperbolic
  geometry}.
\newblock {\em Bulletin of the A.M.S.}, 6(3):357--382, may 1982.

\bibitem[Tsu09]{Tsukamoto2009}
Tatsuya Tsukamoto.
\newblock The almost alternating diagrams of the trivial knot.
\newblock {\em J. Topol.}, 2(1):77--104, 2009.

\bibitem[Vaf15]{Vafaee2013}
Faramarz Vafaee.
\newblock On the knot {F}loer homology of twisted torus knots.
\newblock {\em Int. Math. Res. Not. IMRN}, (15):6516--6537, 2015.

\end{thebibliography}
